\documentclass{amsart}
\usepackage{graphicx,amsfonts,mathrsfs,amssymb,tikz-cd} 
\usepackage{xfrac,comment}
\usepackage[normalem]{ulem}

\usepackage[
  colorlinks=true,
  citecolor=green,  
  linkcolor=blue,  
  urlcolor=black    
]{hyperref}  

\newcommand{\A}{\mathbb{A}}
\newcommand{\C}{\mathbb{C}}
\newcommand{\D}{\mathbb{D}}

\newcommand{\N}{\mathbb{N}}
\newcommand{\bk}{\mathbf{k}}
\newcommand{\calC}{\mathcal{C}}

\newcommand{\I}{\mathcal{I}}

\newcommand{\End}{\mathrm{End}}

\newcommand{\Ext}{\mathrm{Ext}}
\newcommand{\Tor}{\mathrm{Tor}}
\newcommand{\HH}{\mathrm{HH}}
\newcommand{\HC}{\mathrm{HC}}

\theoremstyle{plain}
\newtheorem{Thm}{Theorem}[section]

\newtheorem{Cor}[Thm]{Corollary}
\newtheorem{Prop}[Thm]{Proposition}
\newtheorem{Thm-Def}[Thm]{Theorem-Definition}

\theoremstyle{definition}
\newtheorem{Def}[Thm]{Definition}
\newtheorem{Ex}[Thm]{Example}

\numberwithin{equation}{section}

\newcommand{\multiset}[2]{\genfrac{\{}{\}}{0pt}{1}{#1}{#2}}

\newcommand{\can}[1]{\textcolor{red}{#1 } }

\title{Rigidity of  higher Auslander algebras of type $\mathbb{A}$
}

\date{}

\begin{document}

\begin{abstract}
    We show that the Hochschild cohomology of higher Auslander algebras of type $\mathbb{A}$ vanishes in all positive degrees and is one-dimensional in degree zero. As a consequence,  these algebras are rigid and intrinsically formal.
\end{abstract}

\author{Sibylle Schroll}
\address{Department Mathematik, Universit\"at zu K\"oln, Weyertal 86-90, 50931 K\"oln, Germany}
\email{schroll@math.uni-koeln.de}

\author{Andrea Solotar}
\address{  IMAS-CONICET and Departamento de Matemática, Facultad de Ciencias Exactas y Naturales, Universidad de Buenos Aires, Pabellon I, Ciudad Universitaria, Buenos Aires, 1428, Argentina and Guangdong Technion Israel Institute of Technology, Shantou, Guangdong Province, China.} 
\email{asolotar@dm.uba.ar }

\author{Can Wen}
\address{School of Mathematical Sciences, Laboratory of Mathematics and Complex Systems, Beijing Normal University, Beijing 100875, China and Department Mathematik, Universit\"at zu K\"oln, Weyertal 86-90, 50931 K\"oln, Germany.
}
\email{cwen@mail.bnu.edu.cn}

\maketitle

\section{Introduction}

Let $\bk$ be a field and $A$ be a finite dimensional $\bk$-algebra. The classical Auslander correspondence \cite{Auslander71} is given by the bijection 
\[
\begin{array}{ccc}
\left\{\hspace{-0.2em}
\begin{array}{c}
   \text{finite dimensional $\bk$-algebras} \\
   \text{of finite representation type}
\end{array}
\hspace{-0.2em}\right\}
&
\longleftrightarrow
&
\left\{\hspace{-0.2em}
\begin{array}{c}
   \text{finite dimensional $\bk$-algebras $\Gamma$} \\
   \text{with $\operatorname{gl.dim}\Gamma\le 2\le \operatorname{dom.dim}\Gamma$}
\end{array}
\hspace{-0.2em}\right\}
\\[1.2em]
A
&
\longmapsto
&
\Gamma_A := \End_A(M)
\end{array}
\]
where both classes of algebras are considered up to Morita equivalence, and $M$ is an additive generator of the module category mod-$A$ of $A$. The algebra $\Gamma_A$ is called the {\it Auslander algebra} of $A$. 
In \cite{Iyama07a}, Iyama introduced the class of {\it higher Auslander algebras} in order to generalise the Auslander correspondence to higher dimensions. 
Following \cite{Iyama07a}, the $d$-dimensional Auslander correspondence can be described as the following bijection up to Morita equivalence
\[
\begin{array}{ccc}
\left\{\hspace{-0.2em}
\begin{array}{c}
   \text{finite maximal $d$-orthogonal} \\
   \text{subcategories $\calC$ of mod-$A$}
\end{array}
\hspace{-0.2em}\right\}
&
\longleftrightarrow
&
\left\{\hspace{-0.2em}
\begin{array}{c}
   \text{finite dimensional $\bk$-algebras $\Gamma$} \\
   \text{with $\operatorname{gl.dim}\Gamma\le d\le \operatorname{dom.dim}\Gamma$}
\end{array}
\hspace{-0.2em}\right\}
\\[1.2em]
\calC
&
\longmapsto
&
\Gamma_A := \End_A(M)
\end{array}
\]
where $M$ is an additive generator of $\calC$. It was shown in \cite{Iyama07b} that Auslander-Reiten theory is a special case for $d=2$ of this $d$-dimensional Auslander-Reiten theory.

We now focus on 
the class of {\it higher Auslander algebras of (Dynkin) type $\A$}, originally introduced in \cite{Iyama11}, and further studied  in \cite{OppermannThomas12,JassoKulshammer19,HerschendJorgensen21,DyckerhoffJassoLekili21}.
This class 
has a rich combinatorial structure and has been intensively investigated in recent years, see, for example \cite{HerschendIyama11,IyamaOppermann11,DyckerhoffJassoWalde19}, where they are also referred to as {\it $d$-representation finite algebras of type $\A$}. 

On the other hand, as shown in \cite{DyckerhoffJassoLekili21}, higher Auslander algebras of type $\A$ are  also related to  symplectic geometry as follows. 
Let $\mathcal{W}(\mathrm{Sym}^d(\D),\Lambda_n^{(d)})$ be the ($\mathbb{Z}$-graded) partially wrapped Fukaya category of the $d$-fold {\it symmetric product} of the $2$-dimensional unit disk with finitely many stops on its boundary. That is, let $\D \subset \C$ be the closed unit disk and let $\Lambda_n\subset \partial\D$ be the subset of $(n+1)$-th roots of unity and let $\Lambda_n^{(d)}=\bigcup_{p\in \Lambda_n}p+\mathrm{Sym}^{d-1}(\D)$ be the stops of the  $d$-th symmetric power $\mathrm{Sym}^d(\D)$ of $\D$.
Then there is a quasi-equivalence of triangulated $A_{\infty}$-categories \[ \mathrm{Perf}(A_{n,d})\xrightarrow{\quad\simeq\quad} \mathcal{W}(\mathrm{Sym}^d(\D),\Lambda_n^{(d)}) \]
 where $ \mathrm{Perf}(A_{n,d})$ denotes the perfect derived $A_{\infty}$-category of the higher Auslander algebra $A_{n,d}$. 
 
Since the global dimension of $A_{n,d}$ at most $d$, it follows that  $\HH^i(A_{n,d})=0$ and $\HH_i(A_{n,d})=0$, for each $i>d$. It is then natural to ask what happens in degrees less than or equal to $d$. 

In this short note, we  answer this question: the zero-th Hochschild cohomology is isomorphic to the base field  and the Hochschild cohomology vanishes in all positive degrees, generalising the analogous result for Auslander algebras in \cite{Happel90} to the case of higher Auslander algebras of type $\A$. We also deduce from the shape of the quiver of these algebras that they have Hochschild homology concentrated in degree zero and that their cyclic homology is zero in all odd degrees. More precisely, we show
the following.

\begin{Thm}
    Let $n, d \geq 1$ be integers and let $A_{n,d}$ be a higher Auslander algebra of type $\A$. Set $C^d_n = \frac{n!}{d!(n-d)!}$. Then the following hold

    \begin{enumerate}
        \item   the Hochschild cohomology of $A_{n,d}$ is
        \[\HH^i(A_{n,d})=\begin{cases}
        \bk & \text{ if $i=0$},\\ 0 & \text{ if $i\ge 1$};
    \end{cases}\]
    \item the Hochschild homology of $A_{n,d}$ is
    \[\HH_i(A_{n,d})=\begin{cases}
        \bk^{C_n^d} & \text{if $i=0$},\\ 0 & \text{if $i\ge 1$};
    \end{cases}\]
    
    \item the cyclic homology of $A_{n,d}$ is  \[\HC_i(A_{n,d})=\begin{cases}
        \bk^{C_n^d} & \text{if $i$ is even},\\ 0 & \text{if $i$ is odd}.
    \end{cases}\]
    \end{enumerate}
\end{Thm}

We conclude that higher Auslander algebras of type $\A$  are  rigid, that is, they do not admit non-trivial deformations and furthermore, the absence of non-trivial Hochschild cohomology in positive degrees also shows that these algebras are  intrinsically formal.

The next natural step up in complexity in terms of symmetric products is to consider symmetric products of an annulus with stops. This is precisely what is done in \cite{LuWang26}, where the partially wrapped Fukaya category of the two-fold symmetric product of an annulus with stops is considered. It is shown that the dg endomorphism algebra of a specific generator is formal exactly
when one of the two boundaries contains only one stop and to have a non-trivial $A_\infty$-structure in all other cases. Furthermore, the Hochschild cohomology is calculated and found to be non-trivial in higher degrees (including degree 2), showing that these $A_\infty$-algebras, contrary to the disc case, are no longer rigid but admit non-trivial $A_\infty$-deformations.

\section{Preliminaries}
In this section, we recall some definitions and basic facts which are necessary for computing of the Hochschild (co)homology of higher Auslander algebras of type $\A$. 

\subsection{Incidence algebras}\label{sec:inci-alg}

An {\it incidence algebra} $\I({\bf n})$ is a subalgebra of the algebra $M_n(\bk)$ of square matrices over $\bk$ with elements $(x_{ij})\in M_n(\bk)$ satisfying $x_{ij}=0$ if $i\nleq j$, for some partial order $\leq$ defined in the {\it poset} (partially ordered set) ${\bf n}=\{1,\cdots,n\}$. 

Any incidence algebra can be viewed as a {\it quiver algebra}, that is, a finite dimensional algebra isomorphic to the quotient of a path algebra modulo an admissible ideal. This equivalence between finite posets and ordered quivers can be described as follows. 

Let $Q$ be a finite quiver with a set of vertices $Q_0$ and a set of arrows $Q_1$. Recall that an {\it ordered quiver} is a finite quiver without oriented cycles and such that for each arrow $i \xrightarrow{\alpha} j \in Q_1$ there is no oriented path other than $\alpha$ from $i$ to $j$. For each ordered quiver $Q$, by setting a partial order on the set $Q_0$ of vertices of $Q$ as $i\geq j$ if and only if there exists an oriented path from $i$ to $j$, we can view $Q_0$ as a finite poset. Conversely, if $Q_0$ is a finite poset then we can construct an ordered quiver $Q$ as the Hasse diagram of $Q_0$.

Let $\bk Q$ be the path algebra of $Q$ and $I$ be the {\it parallel ideal} of $\bk Q$, that is, $I$ is the two-sided ideal of $\bk Q$ generated by all the differences $p-q$ where $p,q$ are parallel paths. We call the algebra $A=\bk Q/I$ the {\it incidence algebra} of the poset associated to the ordered quiver $Q$.

\subsection{Higher Auslander algebras of type $\A$}

In this subsection we give several definitions the higher Auslander algebras of type $\mathbb{A}$. We begin with the following definition given in \cite{DyckerhoffJassoLekili21}. 

Given natural numbers  $n$ and $d$, we define the poset 
\[
\multiset{{\bf n}}{d}=\{\Sigma=\{\sigma_1,\cdots,\sigma_d\}\in \N^d\mid 1\le \sigma_1\le\cdots\le \sigma_d\le n\}
\]
to be the poset of $d$-element multisubsets of ${\bf n}=\{1,\cdots,n\}$, where $\Sigma\le \Sigma'$ if for each $1\le i\le d$ the inequality $\sigma_i\le\sigma'_i$ holds.  We can view $\Sigma\in \multiset{\bf n}{d}$ as a {\it monotonic function} 
\[\begin{array}{ccc}
    f:\{1,\cdots,d\} & \longrightarrow & \{1,\cdots,n\} \\
     i & \mapsto & \sigma_i
\end{array}\]
with $\{f(1),\cdots,f(d)\} \in \multiset{\bf n}{d}$. In this sense, $\multiset{\bf n}{d}$ is identified with the poset of monotonic functions equipped with the usual partial order. That is, \[\multiset{\bf n}{d}=\{f:{\bf d}\to {\bf n} \mid f \text{ is a monotonic function}\}.\]
Let $\multiset{\bf n}{d}^{\mathrm{b}}\subseteq \multiset{\bf n}{d}$  be the subset whose elements  $\Sigma\in\multiset{\bf n}{d}$ are such that there exists an index $1\le i<d$ with $\sigma_i=\sigma_{i+1}$. Let 
$\binom{\bf n}{d}$ be the complement of $\multiset{\bf n}{d}^{\mathrm{b}}$ in $\multiset{\bf n}{d}$. We remark that the subset $\binom{\bf n}{d}$ can be identified with the subposet of $\multiset{\bf n}{d}$ spanned by the {\it injective} monotonic functions.  

Now, we are ready to give the first definition of  {\it higher Auslander algebra of type $\A$}. 

\begin{Def}\label{def1}
    Set  
\begin{equation*}
    A_{n,d}:=\frac{\bigoplus_{\Sigma\le\Gamma\in\multiset{\bf n}{d}}\ \bk f_{\Gamma\Sigma}}{\langle f_{\Delta\Delta}\mid \Delta\in \multiset{\bf n}{d}^{\mathrm{b}} \rangle}
\end{equation*}
with multiplication of basis elements given by \[ f_{\Delta\Gamma'}\cdot f_{\Gamma\Sigma}=\begin{cases}
    f_{\Delta\Sigma} & \text{ if $\Gamma=\Gamma'$,}\\ 0 & \text{otherwise.}
\end{cases}\]
\end{Def}

Alternatively, we can define  $A_{n,d}$ as the quotient of the incidence $\bk$-algebra of the poset $\multiset{\bf n}{d}$ by the two-sided ideal generated by the idempotents $f_{\Sigma\Sigma}$, for all $\Sigma\in \multiset{\bf n}{d}^b$. It follows directly from the definition that the $\bk$-algebra $A_{n,d}$ vanishes if $n<d$ and that it is isomorphic to the base field $\bk$ if $n=d$. For $2<d\le n$, it was shown in \cite{Iyama11} that  \[\operatorname{gl.dim}A_{n,d}\le d\le \operatorname{dom.dim}A_{n,d}.\] Moreover, it is straightforward to see that $A_{n,1}$ is the path algebra of the linearly oriented $\mathbb{A}_n$-quiver. Furthermore, the algebras $A_{n,d}$ can also be defined recursively, namely,  $A_{n+1,d+1}$ is the Auslander algebra of $A_{n,d}$.

As pointed out in \cite{DyckerhoffJassoLekili21}, the above definition of {\it higher Auslander algebras of type $\A$} is equivalent to the definition in \cite{OppermannThomas12}. For completeness we recall it here. 

\begin{Def}\label{def2}
    The {\it $d$-dimensional Auslander algebra of type $\A_{n-d+1}$} is the (ungraded) $\bk$-algebra with underlying $\bk$-module 
    \begin{equation}
        A_{n,d}=\frac{\bigoplus_{\Sigma\le\Gamma\in\binom{\bf n}{d}}\ \bk f_{\Gamma\Sigma}}{\langle f_{\Gamma\Sigma}\mid \exists\ i:\gamma_i\ge \sigma_{i+1} \rangle}
    \end{equation}
    equipped with the product \[ f_{\Delta\Gamma'}\cdot f_{\Gamma\Sigma}=\begin{cases}
    f_{\Delta\Sigma} & \text{ if $\Gamma=\Gamma'$,}\\ 0 & \text{otherwise.}
\end{cases}\]
It follows that $A_{n,d}$ is a monomial quotient of the incidence algebra of the poset $\binom{\bf n}{d}$.
\end{Def}

We now recall yet another description of higher Auslander algebras of type $\mathbb{A}$, namely, in terms of quotients of path algebras of quivers as given in \cite{HerschendJorgensen21}. 

\begin{Def}\label{def3}
    Let $r$ and $s$ be two integers such that $r\ge 2$ and $s\ge 1$. 

Define the quiver $Q^{r,s}$ as follows \begin{align*}
        Q^{r,s}_0 & :=\{x=(x_0,x_1,\cdots,x_s)\mid 1\le x_0<x_1<\cdots<x_s\le r+s \}.\quad \end{align*}
Now, for $1\leq k \leq r+s$, set $\sigma_k: Q_0^{r,s} \to Q_0^{r,s}$ to be the partial function defined by 
    \[\sigma_k(x)=y \text{\ with\ } y_i=\begin{cases}
        x_i+1 & \text{if $k=x_i$},\\ x_i & \text{if $k\neq x_i$}.
    \end{cases}\]
Then    \begin{align*} Q^{r,s}_1 & :=\{\alpha_k^x:x\to \sigma_k(x)\mid 1\le k\le r+s \text{ and } x\in Q^{r,s}_0 \text{ such that } \sigma_k(x) \text{ is defined} \}.
    \end{align*} 

     Let $I_{r,s}$ be the ideal of the path algebra $\bk Q^{r,s}$ generated by the following relations 
    \[\rho_{kl}^x=\begin{cases}
        \alpha_l^{\sigma_k(x)}\alpha_k^x-\alpha_k^{\sigma_l(x)}\alpha_l^x & \text{if both $\sigma_k(x)$ and $\sigma_l(x)$ are defined}, \\ \alpha_l^{\sigma_k(x)}\alpha_k^x & \text{if $\sigma_k(x)$ is defined but $\sigma_l(x)$ is not defined}, \\ \alpha_k^{\sigma_l(x)}\alpha_l^x & \text{if $\sigma_l(x)$ is defined but $\sigma_k(x)$ is not defined}.
    \end{cases}\] for $1\le k,l\le r+s$ and $x\in Q_0^{r,s}$.
    
    \medskip
    
     Set  $A_r^s = \bk Q^{r,s}/I_{r,s}$.
\end{Def}

A straightforward comparison between Definitions \ref{def2} and \ref{def3} gives rise to the following proposition where for completeness we include a short proof. 
\begin{Prop}\label{def-equ}\label{prop:equ-def}
    Let $n\ge 2$ and $d\ge 1$ be two integers. With the above notation,  $A_{n,d}$ is isomorphic to $A_{n-d+1}^{d-1}=\bk Q^{n-d+1,d-1}/I_{n-d+1,d-1}$. 
\end{Prop}

\begin{proof}
By Definition \ref{def2}, we can view the quiver $Q$ of $A_{n,d}$ as follows: the vertices are given by the tuples $\Sigma$ in the poset $\binom{\bf n}{d}$, and the arrows are given by $f_{\Gamma\Sigma}:\Sigma\to\Gamma$ for $\Sigma\le\Gamma$ such that there is no $\Delta$ satisfying $\Sigma\le\Delta\le\Gamma$ in $\binom{\bf n}{d}$. Clearly, $Q_0$ is in bijection with $Q_0^{n-d+1,d-1}$. Moreover, we can identify $Q_1$ with $Q_1^{n-d+1,d-1}$. It suffices to show they have the same relations. In fact, with the above identification, we can view the multiplication law $f_{\Delta\Gamma'}\cdot f_{\Gamma\Sigma}$ in Definition \ref{def2} as the concatenation of paths in its associated quiver $Q$. This  gives rise to the non-monomial relations of the ideal $I_{n-d+1,d-1}$ in Definition \ref{def3}, whereas the  ideal \sloppy $\langle f_{\Gamma\Sigma}\mid \exists\ i:\gamma_i\ge \sigma_{i+1} \rangle$  in Definition \ref{def2} corresponds to the ideal generated by the zero relations  of the ideal $I_{n-d+1,d-1}$ in Definition \ref{def3}. We thereby get the statement. 
\end{proof}

Let $J_{m,n}$ be a set consisting of non-decreasing sequences of length $m$. That is, \[J_{m,n}:=\{(a_1,\cdots,a_m)\mid 0\le a_1\le \cdots\le a_m\le n \}. \] Note that, $J_{m,n}$ can be viewed as a poset equipped with termwise comparison. The following theorem shows that these posets are closely related to the higher Auslander algebras of type $\A$.

\begin{Thm}{\rm (\cite[Theorem E]{Gottesman26})}\label{thm:der-equ}
    The incidence algebra of the poset $J_{n-d,d}$ is derived equivalent to the higher Auslander algebra $A_{n,d}$.
\end{Thm}

We illustrate the above definitions in an example. 

\begin{Ex}\label{eg}
Let $n=5$ and $d=2$. 
\begin{itemize}
\item $A_{5,2}$ defined in Definition (\ref{def1}). There are $|\multiset{\bf n}{d}|=C_{n+d-1}^d=15$ vertices.
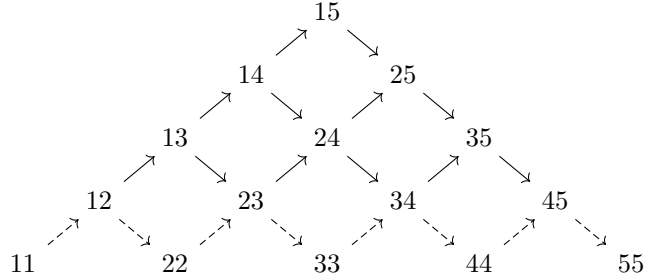
\[\begin{tikzcd}[row sep=1em, column sep=1em]
&&&& 15 \\
&&& 14 && 25 \\
&& 13 && 24 && 35 \\
& 12 && 23 && 34 && 45 \\
11 && 22 && 33 && 44 && 55
\arrow[from=1-5, to=2-6]
\arrow[from=2-4, to=1-5]
\arrow[from=2-4, to=3-5]
\arrow[from=2-6, to=3-7]
\arrow[from=3-3, to=2-4]
\arrow[from=3-3, to=4-4]
\arrow[from=3-5, to=2-6]
\arrow[from=3-5, to=4-6]
\arrow[from=3-7, to=4-8]
\arrow[from=4-2, to=3-3]
\arrow[dashed, from=4-2, to=5-3]
\arrow[from=4-4, to=3-5]
\arrow[dashed, from=4-4, to=5-5]
\arrow[from=4-6, to=3-7]
\arrow[dashed, from=4-6, to=5-7]
\arrow[dashed, from=4-8, to=5-9]
\arrow[dashed, from=5-1, to=4-2]
\arrow[dashed, from=5-3, to=4-4]
\arrow[dashed, from=5-5, to=4-6]
\arrow[dashed, from=5-7, to=4-8]
\end{tikzcd}
\]

Then $A_{5,2}$ is the quotient of the path algebra of the above quiver by the ideal generated by all commutativity relations and the idempotents corresponding to the vertices labelled $kk$. Note that the latter correspond to the idempotents $f_{KK}$, for  $K\in\multiset{\bf n}{d}^{\mathrm{b}}$.

\item $A_{5,2}$ in Definition \ref{def2}. The number of vertices is given by $|\binom{\bf n}{d}|=C_n^d=10$.
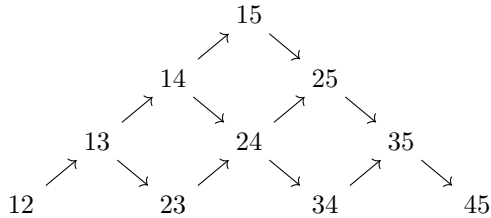
\[\begin{tikzcd}[row sep=1em, column sep=1em]
	&&& 15 &&& \\
	&& 14 && 25 \\
	& 13 && 24 && 35 \\
	12 && 23 && 34 && 45
	\arrow[from=1-4, to=2-5]
	\arrow[from=2-3, to=1-4]
	\arrow[from=2-3, to=3-4]
	\arrow[from=2-5, to=3-6]
	\arrow[from=3-2, to=2-3]
	\arrow[from=3-2, to=4-3]
	\arrow[from=3-4, to=2-5]
	\arrow[from=3-4, to=4-5]
	\arrow[from=3-6, to=4-7]
	\arrow[from=4-1, to=3-2]
	\arrow[from=4-3, to=3-4]
	\arrow[from=4-5, to=3-6]
\end{tikzcd}\]

Then $A_{5,2}$ is the quotient of the path algebra of the above quiver by the ideal generated by all commutativity relations and the monomial relations corresponding to the paths of length two with two vertices in the bottom row. 

\item $A_4^1$ given by Definition \ref{def3}. There are $C_n^d=10$ vertices in $Q^{4,1}$.
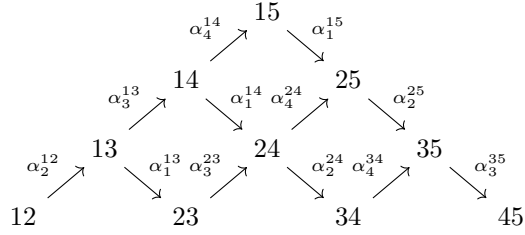
\[\begin{tikzcd}[row sep=1.2em, column sep=1.2em]
	&&& 15 &&& \\
	&& 14 && 25 \\
	& 13 && 24 && 35 \\
	12 && 23 && 34 && 45
	\arrow["{\alpha_1^{15}}", from=1-4, to=2-5]
	\arrow["{\alpha_4^{14}}", from=2-3, to=1-4]
	\arrow["{\alpha_1^{14}}", from=2-3, to=3-4]
	\arrow["{\alpha_2^{25}}", from=2-5, to=3-6]
	\arrow["{\alpha_3^{13}}", from=3-2, to=2-3]
	\arrow["{\alpha_1^{13}}", from=3-2, to=4-3]
	\arrow["{\alpha_4^{24}}", from=3-4, to=2-5]
	\arrow["{\alpha_2^{24}}", from=3-4, to=4-5]
	\arrow["{\alpha_3^{35}}", from=3-6, to=4-7]
	\arrow["{\alpha_2^{12}}", from=4-1, to=3-2]
	\arrow["{\alpha_3^{23}}", from=4-3, to=3-4]
	\arrow["{\alpha_4^{34}}", from=4-5, to=3-6]
\end{tikzcd}\]

Then $A_4^1$ is the quotient of the path algebra of the above quiver by the ideal generated by all commutativity relations and the monomial relations corresponding to the paths of length two with two vertices in the bottom row. 

\item The incidence algebra $\I(J_{3,2})$ of poset $J_{3,2}$. There are $C_n^d=10$ vertices in the Hasse diagram of $J_{3,2}$.
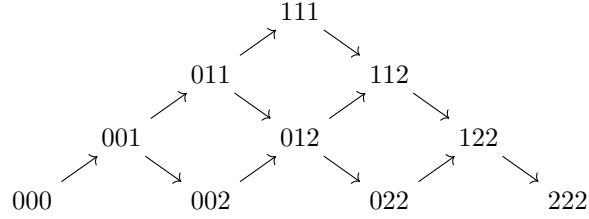
\[\begin{tikzcd}[row sep=1em, column sep=1em]
	&&& 111 &&& \\
	&& 011 && 112 \\
	& 001 && 012 && 122 \\
	000 && 002 && 022 && 222
	\arrow[from=1-4, to=2-5]
	\arrow[from=2-3, to=1-4]
	\arrow[from=2-3, to=3-4]
	\arrow[from=2-5, to=3-6]
	\arrow[from=3-2, to=2-3]
	\arrow[from=3-2, to=4-3]
	\arrow[from=3-4, to=2-5]
	\arrow[from=3-4, to=4-5]
	\arrow[from=3-6, to=4-7]
	\arrow[from=4-1, to=3-2]
	\arrow[from=4-3, to=3-4]
	\arrow[from=4-5, to=3-6]
\end{tikzcd}\]

Then $\I(J_{3,2})$ is the quotient of the path algebra of the above quiver by the ideal generated by all commutativity relations. Note that in this case there are no monomial relations in the ideal. 
\end{itemize} 

Thus we can clearly see in this example, that the first three  algebras are isomorphic while the fourth one is not isomorphic but it is derived equivalent to them. 
\end{Ex}

\section{Hochschild (co)homology of higher Auslander algebras of type A}

Recall that the Hochschild cohomology of an associative algebra $A$ with coefficients in an $A$-bimodule $M$ was originally introduced in \cite{Hochschild45} and later proved to be isomorphic to $\Ext_{A\otimes A^{op}}^*(A,M)$. As a consequence, the Hochschild (co)homology of a $\bk$-algebra $A$
can be computed from any projective resolution of $A$ over its enveloping algebra $A^e=A\otimes A^{op}$.
On the other hand, the Hochschild homology of $A$ with coefficients in $M$ is isomorphic to
$\Tor^{A^e}_*(M,A)$.

We are now ready to state  our main theorem.

\begin{Thm}\label{thm:Hoch-cohomology}
    Let $A_{n,d}$ be a higher Auslander algebra of type $\A$.  Then \[\HH^i(A_{n,d})=\begin{cases}
        \bk & \text{ if $i=0$},\\ 0 & \text{ if $i\ge 1$}.
    \end{cases}\]
\end{Thm}

\begin{proof}
   First notice that the quiver of $A_{n,d}$ contains no cycles which implies that  \sloppy ${\rm HH}^0(A_{n,d})=\bk$. 
    Furthermore, since Hochschild cohomology is invariant under derived equivalence, by 
    Theorem \ref{thm:der-equ} the algebra $A_{n,d}$ and the incidence algebra $\I(J_{n-d,d})$ of the poset $J_{n-d,d}$ have  isomorphic Hochschild cohomology. 
   By Theorem 2.2 in \cite{GaticaRedondo01}, if the set of vertices of an incidence algebra $A=\bk Q/I$ has a unique maximal or minimal element then $\HH^i(A)=0$, for all $i\ge 1$. 
   Recall that we can identify the incidence algebra $\I(\Sigma)$ of a poset $\Sigma$ with a quiver algebra $\bk Q/I$, where $Q$ is the Hasse diagram of $\Sigma$. So $Q_0$ has a unique maximal or minimal element if and only if $\Sigma$ does.
   Hence for an incidence algebra  $\I(\Sigma)$  of a poset $\Sigma$, if  $\Sigma$ has a unique maximal or minimal element, then we also have $\HH^i(\I(\Sigma))=0$, for all $i\ge 1$.
    Thus, as  
    the poset $J_{n-d,d}$ has a unique minimal element $(d,\cdots,d)$ and a unique maximal element $(0,\cdots,0)$, the result follows.
\end{proof}

Since the quiver of $A_{n,d}$ has no oriented cycles, all higher Hochschild homology groups are zero and the only non-zero homology is in degree zero \cite{Cibils86}, see also \cite{Han06}. 

\begin{Thm}\label{thm:Hoch-homology}
Let $A_{n,d}$ be a higher Auslander algebra of type $\A$. Denote by $C_n^d$  the binomial coefficient. Then 
 the Hochschild homologies of $A_{n,d}$ are given by \[\HH_i(A_{n,d})=\begin{cases}
        \bk^{C_n^d} & \text{if $i=0$},\\ 0 & \text{if $i\ge 1$};
    \end{cases}\]
\end{Thm}

Using  the {\it homological Connes' periodicity exact sequence}, also called the SBI sequence, in \cite{Cibils86}, the cyclic homology of an algebra of a quiver with relations, where the quiver has no oriented cycles, is given. Applying this result, the cyclic homology of higher Auslander algebras is as follows.

\begin{Cor}\label{cor:cyclic-homology}
Let $A_{n,d}$ be a higher Auslander algebra of type $\A$. Denote by $C_n^d$  the binomial coefficient. Then 
 the cyclic homologies of $A_{n,d}$ are given by \[\HC_i(A_{n,d})=\begin{cases}
        \bk^{C_n^d} & \text{if $i$ is even},\\ 0 & \text{if $i$ is odd}.
    \end{cases}\]
   
\end{Cor}

An immediate consequence of the above results is that  the higher Auslander algebras of type $\A$ satisfy Happel's question and Han's conjecture.

\section*{ Acknowledgements}
The last author is supported by the China Scholarship Council (CSC). The last author gratefully acknowledges the hospitality of the University of Cologne where this work has been completed. 

\bibliographystyle{amsplain}
\bibliography{refs}

\end{document}